\documentclass[11pt,a4paper]{article}

\usepackage[utf8]{inputenc}
\usepackage[T1]{fontenc}
\usepackage[english]{babel}
\usepackage{amssymb,amsmath,amsthm}
\usepackage{xspace}
\usepackage[margin=2.8cm]{geometry}

\theoremstyle{plain}
\newtheorem{theorem}{Theorem}
\newtheorem{corollary}{Corollary}
\newtheorem{proposition}{Proposition}
\newtheorem{lemma}{Lemma}

\theoremstyle{definition}
\newtheorem{remark}{Remark}
\newtheorem{example}{Example}

\newcommand{\ds}{\displaystyle}

\newcommand{\calA}{ \mathcal{A} }
\newcommand{\calC}{ \mathcal{C} }
\newcommand{\calJ}{ \mathcal{J} }
\newcommand{\calK}{ \mathcal{K} }
\newcommand{\calL}{ \mathcal{L} }

\newcommand{\pt}{ {\partial}_{t} }
\newcommand{\px}{ {\partial}_{x} }
\newcommand{\gradx}{ \nabla }
\newcommand{\deltax}{\partial^{2}_{xx}}

\newcommand{\abs}[1]{ \left|#1\right| }
\newcommand{\rank}{ \mathrm{rank} \,}
\newcommand{\vect}{ \mathrm{span} }

\newcommand{\field}[1]{\ensuremath{\mathbb{#1}}}
\newcommand{\C}{\field{C}\xspace}
\newcommand{\R}{\field{R}\xspace}
\newcommand{\mat}[1]{ {\mathcal{M}}_{#1}( \mathbb{R} ) }

\newcommand{\notc}[2]{ { \left(#1\right) }_{#2} }
\newcommand{\nota}[2]{ {\left( #1 \right)}_{#2}^{\pm} }
\newcommand{\notb}[2]{ \mathsf{ {#1}_{#2} } }
\newcommand{\notd}[2]{ {\left(\left(#1\right)\right)}_{#2} }
\newcommand{\kalm}[3]{ {\left[\, #1 \, : \, #2 \, \right]}_{#3} }

\title{Null-controllability for some linear parabolic systems\\
with controls acting on different parts of the domain\\
and its boundary}

\author{Guillaume Olive\thanks{LATP, UMR 6632, Aix-Marseille Universit\'e,
Technop\^ole Ch\^ateau-Gombert, 39, rue F. Joliot-Curie,
13453 Marseille Cedex 13, France. E-mail: \texttt{golive@cmi.univ-mrs.fr}}}

\date{}

\begin{document}

\maketitle

\begin{abstract}
In this work we study the null-controllability properties of linear parabolic
systems with constant coefficients in the case where several controls are
acting on different distributed subdomains and/or on the boundary.
We prove a Kalman rank condition in the one-dimensional case.
In the case where only distributed controls are considered we also establish
related results such as a Carleman estimate.
\end{abstract}

\noindent\textbf{Keywords:}
Kalman rank condition; Boundary controllability; Distributed controllability;
Carleman estimate

\medskip
\section{Introduction}

Let $n \in \mathbb{N}^*$, $n_D, n_B \in \mathbb{N}$, be respectively the number of equations, the number of distributed controls and the number of boundary controls we will consider. Let $\Omega \subset \mathbb{R}^N$ ($N \in \mathbb{N}^*$) be a bounded connected open set with boundary $\partial \Omega$ regular enough. For every $T> 0$ we denote $Q_T=(0,T) \times \Omega$ and $\Sigma_{T}=(0,T) \times \partial \Omega$. Let $\omega_1,\ldots, \omega_{n_D}$ be given non empty open subsets of $\Omega$ (possibly disjoint) and let $\Gamma_1, \ldots, \Gamma_{n_B}$ be given non empty open subsets of $\partial \Omega$.
We consider the following type of $n\times n$ parabolic system:
\begin{equation}\label{syst plrs control}
\left\{
\begin{array}{ll}
&\pt y = \Delta y + Ay + D_1 u_1(t,x) 1_{\omega_1}(x) +\ldots+ D_{n_D} u_{n_D}(t,x) 1_{\omega_{n_D}}(x) \mbox{ in } Q_{T}, \\
&y=B_1v_1(t,x)1_{\Gamma_1}(x)+\ldots+B_{n_B}v_{n_B}(t,x)1_{\Gamma_{n_B}}(x) \mbox{ on } \Sigma_{T}, \\
\end{array}
\right.
\end{equation}

where $y$ is the state, $A \in \mat{n}$ is a coupling matrix. For every $i \in \{1, \ldots, n_D\}$, $D_i \in \R^n$ and $u_i \in L^2(Q_{T})$ is a \textit{distributed control} acting on $\omega_i$. For every $j \in \{1, \ldots, n_B\}$, $B_j \in \R^n$ and $v_j \in L^2(\Sigma_{T})$ is a \textit{boundary control} acting on $\Gamma_j$.

Let us recall that, for every $T>0$, for every $y_0 \in  L^2(\Omega;\R^n)$, $u_i \in L^2(Q_{T})$, $i \in \{1, \ldots, n_D\}$, and $v_j \in L^2(\Sigma_{T})$, $j \in \{1, \ldots, n_B\}$, there exists a unique solution to (\ref{syst plrs control}) $y \in L^2(Q_{T};\R^n) \, \cap \, C^0([0,T];H^{-1}(\Omega; \R^n))$, defined by transposition, which satisfies $y(0)=y_0$ (see for instance \cite[Appendix]{FCGBdT} for more details).

Let be given $y_0 \in L^2(\Omega;\R^n)$ and $T> 0$, it will be said that system (\ref{syst plrs control}) is \emph{null-controllable on $(0,T)$ from the state $y_0$} if there exists $u_i \in L^2(Q_{T})$ for every $i \in \{1, \ldots, n_D\}$ and there exists $v_j \in L^2(\Sigma_{T})$ for every $j \in \{1, \ldots, n_B\}$ such that the corresponding solution to (\ref{syst plrs control}) with $y(0)=y_0$ satisfies $y(T)=0$.
Let be given $T> 0$, it will be said that system (\ref{syst plrs control}) is \emph{null-controllable on $(0,T)$}  if for every $y_0 \in L^2(\Omega;\R^n)$ system (\ref{syst plrs control}) is null-controllable on $(0,T)$ from the state $y_0$.
It will be said that system (\ref{syst plrs control}) is \emph{null-controllable} if for every $T> 0$ system (\ref{syst plrs control}) is null-controllable on $(0,T)$.

Let us recall that the scalar operator $-\Delta$ with homogeneous Dirichlet boundary condition admits a sequence of eigenvalues $\{\lambda_k\}_{k \in \mathbb{N}^*} \subset {\mathbb{R}}_{+}^{*}$ such that the associated sequence of normalized eigenfunctions $\{\phi_k\}_{k \in \mathbb{N}^*}$ is a Hilbert basis of $L^2(\Omega)$.

\paragraph{Matrices notation}
For any $q, N_1, N_2 \in \mathbb{N}^*$, for any matrix  $A \in \mat{N_1}$, $B \in \mat{N_1 \times N_2}$, we denote $[A|B]$ the matrix whose first columns are those of $A$ and the following ones are those of $B$ and we define
\begin{equation}\label{matrices notations}
\footnotesize 
\begin{array}{c}
\kalm{A}{B}{q}= \left[ B | AB | \cdots | A^{q-1}B \right]
\in \mat{N_1 \times N_2q},
\quad
\notd{B}{q}=\left[\begin{array}{ccccc}
B & 0 & \cdots & \cdots & 0 \\
0 & B & \ddots & & \vdots \\
\vdots & \ddots & \ddots & \ddots & \vdots \\
\vdots &  & \ddots & \ddots &0 \\
0 & \cdots &  \cdots & 0 &B
\end{array} \right]
\in \mat{N_1q \times N_2q}
,
\\
\notc{B}{q}= \left[ \begin{array}{c} B \\ B \\ \vdots \\ \vdots \\ B  \end{array}\right]
\in \mat{N_1q \times N_2}
, \quad
\nota{B}{q}= \left[ \begin{array}{c} - B \\ B \\ \vdots \\ \vdots \\  {(-1)}^{q} B  \end{array}\right]
\in \mat{N_1q \times N_2}
,
\\
\calA_q=\left[\begin{array}{ccccc}
A_1 & 0 & \cdots & \cdots & 0 \\
0 & A_2 & \ddots & & \vdots \\
\vdots & \ddots & \ddots & \ddots & \vdots \\
\vdots &  & \ddots & \ddots &0 \\
0 & \cdots &  \cdots & 0 & A_q
\end{array} \right]
\in \mat{N_1q}
\mbox{ with }
A_k=-\lambda_k I +A
\in \mat{N_1}.
\end{array}
\end{equation}

Note that
\begin{equation}\label{rem rank ))}
\rank{ \notd{ \kalm{A}{B}{n} }{q} }=q \, \rank{ \kalm{A}{B}{n} }, \quad \forall q \in \mathbb{N}^*.
\end{equation}

In controllability theory of linear ordinary differential systems there exists a complete characterization of controllability, this is the so-called \textit{Kalman rank condition} (see for instance \cite[Corollary 1.4.10]{TW}), that is to say, if $A\in \mat{n}$ and $B\in \mat{n \times m}$ ($n,m \in \mathbb{N}^*$), then the linear ordinary differential system $ y^{\prime}=Ay+Bu$ is controllable if and only if
\begin{equation*}
\rank{ \kalm{A}{B}{n} }=n.
\end{equation*}

To give an appropriate condition of this type in the framework of linear parabolic systems has been a subject of several research. Recently in \cite{AKBDGB} a Kalman rank condition has been proved for the distributed null-controllability with only one control region: the authors proved that the system
\begin{equation*}
\left\{
\begin{array}{ll}
&\pt y = \Delta y + Ay + Du_1(t,x) 1_{\omega_1}(x)  \mbox{ in } Q_{T}, \\
\noalign{\smallskip}
&y=0 \mbox{ on } \Sigma_{T}, \\
\noalign{\smallskip}
\end{array}
\right.
\end{equation*}
is null-controllable if and only if
\begin{equation}\label{recall distributed Kalman condition}
\rank{ \kalm{A}{D}{n} }=n,
\end{equation}
see \cite[Theorem 1.1]{AKBDGB} and \cite[Proposition 2.2]{AKBDGB}.
In fact they proved a more general Kalman rank condition for linear parabolic systems with different coefficients in front of the operator $-\Delta$, for more details see \cite{AKBDGB}.

In \cite{FCGBdT} the authors gave a necessary and sufficient condition for the boundary null-controllability in the one-dimensional case and for two equations (see \cite[Theorem 1.1]{FCGBdT}). Through this theorem they also showed that the Kalman rank condition for distributed null-controllability is a necessary condition for the boundary null-controllability but it is not sufficient.
The result of \cite{FCGBdT} has been improved in \cite{AKBGBdT} where the authors proved a new Kalman rank condition for boundary null-controllability of $n \times n$ linear parabolic system (still in the one-dimensional case though), that is: the system
\begin{equation*}
\left\{
\begin{array}{ll}
&\pt y = \deltax y + Ay \mbox{ in } (0,T) \times (0,1), \\
\noalign{\smallskip}
&y(t,0)=B_1v_1(t), \quad y(t,1)=B_2v_2(t) \mbox{ on } (0,T), \\
\noalign{\smallskip}
\end{array}
\right.
\end{equation*}
is null-controllable if and only if
\begin{equation}\label{recall boundary Kalman condition}
\rank{ \kalm{\calA_q}{ \left( \notc{B_1}{q} \middle| \nota{B_2}{q} \right) }{nq} }=nq, \quad \forall q \in \mathbb{N}^*,
\end{equation}
see \cite[Theorem 6.3]{AKBGBdT}.

We also point out reference \cite{AKBDGB2} where the authors worked on the case of regular time-dependent matrices $A=A(t)$ and $D=D(t)$ (with one distributed control) and they proved that the system
\begin{equation*}
\left\{
\begin{array}{ll}
&\pt y = \Delta y + A(t)y + D(t)u_1(t,x) 1_{\omega_1}(x)  \mbox{ in } Q_{T}, \\
\noalign{\smallskip}
&y=0 \mbox{ on } \Sigma_{T}, \\
\noalign{\smallskip}
\end{array}
\right.
\end{equation*}
is null-controllable if
\begin{equation}\label{kalm for t-depend}
\exists t_0 \in [0,T], \quad \rank{\calK(t_0) }=n,
\end{equation}
where
$$\calK(t)=\left[ \calK_0(t) \middle| \ldots \middle| \calK_{n-1}(t) \right]$$
with
$$
\left\{\begin{array}{rcl}
\ds \calK_0(t)&=&D(t) \\
\ds \calK_i(t)&=&\ds A(t)\calK_{i-1}(t)-\frac{d}{dt} \calK_{i-1}(t), \quad \forall i \in \left\{1,\ldots,n-1\right\}
\end{array}\right.
$$
see \cite[Theorem 1.2]{AKBDGB2}.

Finally, let us mention that the null-controllability properties of linear parabolic systems have been also studied in the case of space varying coefficients and one distributed control force. However few results are known, even for the distributed controllability. Sufficient conditions to the distributed null-controllability are given in \cite{AKBD}, \cite{AKBDK}, \cite{GBPG},\cite{Gue} and \cite{dT} for systems of two equations and see \cite{GBdT} for $n \times n$ systems.
To our knowledge \cite{dT} is also the first result using Carleman estimates for two coupled parabolic equations.
Concerning the boundary null-controllability of parabolic systems, let us mention \cite{ABL} where the authors prove a result without any restriction on the dimension (but under some geometric condition, see \cite{ABL} for more details).

In the present work we try to give an overview of the controllability properties of systems like (\ref{syst plrs control}). One of the main task of this work will be to prove a Kalman rank condition for system (\ref{syst plrs control}) which will then generalize the previously known Kalman conditions (\ref{recall distributed Kalman condition}) and (\ref{recall boundary Kalman condition}).

\section{Statements of the results}

\subsection{Main result}

The first and main result of this work concerns system (\ref{syst plrs control}) in the case $N=1$. As a consequence it is equivalent to consider the following system:
\begin{equation}\label{syst2}
\left\{
\begin{array}{ll}
&\pt y = \deltax y + Ay + D_1 u_1(t,x) 1_{\omega_1}(x) +\ldots+ D_{n_D} u_{n_D}(t,x) 1_{\omega_{n_D}}(x) \mbox{ in } Q_{T}, \\
\noalign{\smallskip}
&y(t,0)=B_1^L v_1(t) + \cdots + B_{n_L}^L v_{n_L}(t), \quad y(t,1)=B_1^R w_1(t) + \cdots + B_{n_R}^R w_{n_R}(t) \mbox{ on } (0,T), \\
\noalign{\smallskip}
\end{array}
\right.
\end{equation}
with $n_B \leq n_L+n_R \leq 2n_B$, and where we take $\Omega=(0,1)$ for the sake of simplicity. Let us denote $D=\left[ D_1 | D_2 | \cdots | D_{n_D}\right] \in \mat{n \times n_D}$, $B^{L}=\left[ B^{L}_1 | \cdots | B^L_{n_L}\right]$ and $B^{R}=\left[ B^{R}_1 | \cdots | B^R_{n_R}\right]$.
Then, the result reads:

\begin{theorem}[Kalman rank condition]\label{boundary Kalman condition theorem}
System (\ref{syst2}) is null-controllable if and only if
\begin{equation}\label{boundary Kalman condition}
\rank{ \left[ \quad \kalm{\calA_q}{  \left( \notc{B^L}{q} \middle| \nota{B^R}{q} \right) }{nq} \quad \middle| \quad \notd{ {\left[A:D\right]}_{n} }{q} \quad \right] } = nq, \quad \forall q \in \mathbb{N}^*,
\end{equation}
(where we used the notations introduced above).
\end{theorem}

\begin{remark}\label{remark about rank 1}\label{another characterization of the boundary Kalman condition}

\begin{enumerate}

\item
Theorem \ref{boundary Kalman condition theorem} contains both Kalman conditon (\ref{recall distributed Kalman condition}) for distributed null-controllability and Kalman condition (\ref{recall boundary Kalman condition}) for boundary null-controllability, see (\ref{rem rank ))}).

\item
We can also reformulate condition (\ref{boundary Kalman condition}) as follows:
\begin{equation}\label{boundary Kalman condition 2}
\rank{ \kalm{\calA_q}{ \left( \notc{B^L}{q} \middle| \nota{B^R}{q}\middle| \notd{D}{q} \right) }{nq} } = nq, \quad \forall q \in \mathbb{N}^*.
\end{equation}
This is due to the equalities $\rank{ \kalm{\calA_q}{\notd{D}{q}}{nq} } = \rank{ \notd{\kalm{A}{D}{n}}{q} }$ for all $q \in \mathbb{N}^*$.
This characterization will be used to prove Theorem \ref{boundary Kalman condition theorem}.

\item
Let us observe that condition (\ref{boundary Kalman condition}) is only algebraic. In particular it does not depend on $\omega_1, \ldots, \omega_{n_D}$.

\item
Condition (\ref{boundary Kalman condition}) is checkable thanks to the following fact: to check condition (\ref{boundary Kalman condition}) is equivalent to check it for a particular $q=q_0$ which is such that
\begin{equation}\label{eigenvalues of A}
\mu_i-\mu_j \neq \lambda_k - \lambda_l, \quad \forall k,l \in \mathbb{N}^* \mbox{ with } k>q_0 \mbox{ and } l \neq k, \quad \forall i,j \in \{1,\ldots,n\},
\end{equation}
where $\{\mu_k\}_{k \in \{1,\ldots,n\}} \subset \C$ is the set of the eigenvalues of $A$.
To prove this fact one can adapt the proof of \cite[Corollary 3.3]{AKBGBdT}, by using the characterization (\ref{boundary Kalman condition 2}).
Moreover one can see that such a $q_0$ does always exist, see \cite[Proposition 3.2]{AKBGBdT} for instance.

\end{enumerate}
\end{remark}
Let us illustrate the last item of Remark \ref{remark about rank 1} through the following example:

\begin{example}\label{example KRC}
Let $T>0$ and $\omega \subset (0,1)$ a non empty open subset. Consider the following $3 \times 3$ one-dimensional parabolic system:
\begin{equation}\label{example system}
\left\{
\begin{array}{ll}
&
\left. \begin{array}{l}
\pt y_1 = \deltax y_1 +2y_1 +6y_2 +2y_3, \\
\noalign{\smallskip}
\pt y_2 = \deltax y_2 +4y_1  -2y_3, \\
\noalign{\smallskip}
\pt y_3 = \deltax y_3 +2y_1 -3/2y_2 +2y_3 + u(t,x)1_{\omega}(x)  , \\
\noalign{\smallskip}
\end{array}
\right\}
\mbox{ in } (0,T) \times (0,1), \\
&
\left. \begin{array}{l}
y_1(t,0)=v(t), \quad y_1(t,1)=0, \\
\noalign{\smallskip}
y_2(t,0)=0 \quad y_2(t,1)=0, \\
\noalign{\smallskip}
y_3(t,0)=0 \quad y_3(t,1)=0, \\
\noalign{\smallskip}
\end{array}
\right\}
\mbox{ on } (0,T),
\end{array}
\right.
\end{equation}
so that
$$
A=\left[\begin{array}{ccc}
2 & 6 & 2 \\
4 & 0 & -2 \\
2 & -3/2 & 2 
\end{array} \right],
\quad D=\left[\begin{array}{c} 0 \\ 0 \\ 1 \end{array} \right],
\quad B_L=\left[\begin{array}{c} 1 \\ 0 \\ 0 \end{array} \right],
\quad B_R=0.
$$
We can see that if only one control is acting then this system is not null-controllable. Indeed we have $\rank{ \kalm{A}{D}{3} }=2 \neq 3$ and $\rank{ \kalm{A}{B}{3} }=2 \neq 3$ so the distributed and boundary Kalman conditions fail.
Nevertheless we have $\rank{ \left[ \ \kalm{A}{B}{3} \ \middle| \ \kalm{A}{D}{3} \ \right] }=3$ and the eigenvalues of $A$ are $-5,3$ and $6$ so that condition (\ref{eigenvalues of A}) is satisfied for $q_0=1$ and thus, by the previous remark, condition (\ref{boundary Kalman condition}) is also satisfied.
\end{example}

We make the following remark about this example, this gives a good idea of the proof of Theorem \ref{boundary Kalman condition theorem}:

\begin{remark}

Observe that in fact the matrix $A$ of Example \ref{example KRC} is equivalent to the matrix
$$C=\left[\begin{array}{ccc}
6 & 0 & 0 \\
6 & 0 & 15 \\
-2 & 1 & -2
\end{array} \right]$$
through the following change of basis:
$$P=\left[\begin{array}{ccc}
1 & 0 & 2 \\
0 & 0 & -2 \\
0 & 1 & 2
\end{array} \right]
=\left[ B| D | AD \right] \quad \mbox{ (with the notations of Example \ref{example KRC}) }.$$
As a consequence the null-controllability of system (\ref{example system}) is equivalent to the null-controllability of the system
$$
\left\{
\begin{array}{ll}
&\pt z = \deltax z + 
\left[\begin{array}{ccc}
6 & 0 & 0 \\
6 & 0 & 15 \\
-2 & 1 & -2
\end{array} \right]
z +
\left[\begin{array}{c} 0 \\ 1 \\ 0 \end{array} \right]
u(t,x) 1_{\omega}(x) \mbox{ in } (0,T) \times (0,1), \\
\noalign{\smallskip}
&z(t,0)=
\left[\begin{array}{c} 1 \\ 0 \\ 0 \end{array} \right]
v(t),
\quad
z(t,1)=0 \mbox{ on } (0,T), \\
\noalign{\smallskip}
\end{array}
\right.
$$
And we can see that we can lead the first component of this system to zero at time $T/2$ (for instance) by using only the boundary control $v$. Then, it remains to prove that the system
$$
\left\{
\begin{array}{ll}
&\pt \hat{z} = \deltax \hat{z} + 
\left[\begin{array}{cc}
0 & 15\\
1 & -2
\end{array} \right]
\hat{z}
+ \ds
\left[\begin{array}{c} 1 \\ 0 \end{array}\right]
u(t,x)1_{\omega}(x)
\mbox{ in } \left(\frac{T}{2},T\right) \times (0,1), \\
\noalign{\smallskip}
&\ds \hat{z}(t,0)=0,
\quad
\hat{z}(t,1)=0 \mbox{ on } \left(\frac{T}{2},T\right), \\
\noalign{\smallskip}
\end{array}
\right.
$$
is null-controllable, which can be done by checking the distributed Kalman condition.

\end{remark}

\subsection{More results in the case $n_B=0$}\label{section n_B}

Let us now consider the case without boundary control with an arbitrary space dimension $N$, that is
\begin{equation}\label{syst45}
\left\{
\begin{array}{ll}
&\pt y = \Delta y + Ay + D_1 u_1(t,x) 1_{\omega_1}(x) +\ldots+ D_{n_D} u_{n_D}(t,x) 1_{\omega_{n_D}}(x) \mbox{ in } Q_{T}, \\
\noalign{\smallskip}
&y=0 \mbox{ on } \Sigma_{T}. \\
\noalign{\smallskip}
\end{array}
\right.
\end{equation}
From the proof of Theorem \ref{boundary Kalman condition theorem} we will see that in fact Theorem \ref{boundary Kalman condition theorem} still holds without any restriction on $N$ if we have no boundary controls:

\begin{corollary}\label{Distributed Kalman rank condition theorem}
Let $N \geq 1$ be arbitrary. Then, system (\ref{syst45}) is null-controllable if and only if
\begin{equation}\label{distributed Kalman condition}
\rank{ \kalm{A}{D}{n} } = n.
\end{equation}
\end{corollary}

On the other hand, when the Kalman condition (\ref{distributed Kalman condition}) is not fulfilled it is possible to characterize the states that can be driven to $0$:

\begin{proposition}\label{remark about the controllable states}
Assume that $N \geq 1$ and $\rank{ \kalm{A}{D}{n} } < n$. Then, system (\ref{syst45}) is null-controllable on $(0,T)$ from the state $y_0$ for every $T> 0$ if and only if
$$y_0 \in L^2 \left(\Omega; \vect \kalm{A}{D}{n} \right).$$
\end{proposition}
This can be proved by extending the arguments given in \cite[Theorem 1.5]{AKBDGB}.

We also can also extend the Kalman rank condition for time-dependent matrices (\ref{kalm for t-depend}): let us consider the system
\begin{equation}\label{syst99}
\left\{
\begin{array}{ll}
&\pt y = \Delta y + A(t)y + D_1(t) u_1(t,x) 1_{\omega_1}(x) +\ldots+ D_{n_D}(t) u_{n_D}(t,x) 1_{\omega_{n_D}}(x) \mbox{ in } Q_{T}, \\
\noalign{\smallskip}
&y=0 \mbox{ on } \Sigma_{T}. \\
\noalign{\smallskip}
\end{array}
\right.
\end{equation}
where $A \in \calC^{n-1}\left([0,T];\mat{n}\right)$ and $D_i \in \calC^{n}\left([0,T];\R^n\right)$.

In this case we have
\begin{theorem}\label{t-dep}
If there exists $t_0 \in [0,T]$ such that
\begin{equation}\label{kalm for t-dep}
\exists t_0 \in [0,T], \quad \rank{\calK(t_0) }=n,
\end{equation}
where
$$\calK(t)=\left[ \calK_0(t) \middle| \ldots \middle| \calK_{n-1}(t) \right]$$
with
$$
\left\{\begin{array}{rcl}
\ds \calK_0(t)&=&D(t)=\left[D_1(t) \middle| \ldots \middle| D_{n_D}(t) \right] \\
\ds \calK_i(t)&=&\ds A(t)\calK_{i-1}(t)-\frac{d}{dt} \calK_{i-1}(t), \quad \forall i \in \left\{1,\ldots,n-1\right\}
\end{array}\right.
$$
then the system (\ref{syst99}) is null-controllable at time $T$.
\end{theorem}

This theorem will be proved thanks to a Carleman estimate for cascade systems, see Theorem \ref{Carleman estimate for coupled linear parabolic systems of cascade type theorem} below.

\section{The Kalman rank condition}

Recall that all along this section we assume that $N=1$ and the system considered is (\ref{syst2}). In fact, for the sake of simplicity of the notations we will consider
\begin{equation}\label{syst3}
\left\{
\begin{array}{ll}
&\pt y = \deltax y + Ay + D_1 u_1(t,x) 1_{\omega_1}(x) +\ldots+ D_{n_D} u_{n_D}(t,x) 1_{\omega_{n_D}}(x) \mbox{ in } Q_{T}, \\
\noalign{\smallskip}
&y(t,0)=B_1v_1(t), \quad y(t,1)=B_2v_2(t) \mbox{ on } (0,T). \\
\noalign{\smallskip}
\end{array}
\right.
\end{equation}

\subsection{Some known results}

Before starting the proof of Theorem \ref{boundary Kalman condition theorem} let us recall for convenience the following results:
\begin{proposition}\label{prop Hautus test}
Let be given $A \in \mat{n}$ and $B \in \mat{n \times m}$ ($n,m \in \mathbb{N}^*$). We have
$$\forall V \in \R^n, \quad \left(\forall t \geq 0, \quad B^*e^{tA^*}V=0 \right) \Longrightarrow V=0$$
if and only if
$$\rank{ \kalm{A}{B}{n} }=n$$
if and only if
\begin{equation}\label{recall Hautus test}
\ker (A^*- \theta I) \, \cap \, \ker B^*= \{0\}, \quad \forall \theta \in \mathbb{C}.
\end{equation}
\end{proposition}

Condition (\ref{recall Hautus test}) is the so-called \textit{Hautus test}. For a proof see for instance \cite[Chapter 1]{TW}.
To state the second result we need to define the \textit{adjoint system} of (\ref{syst3}):
\begin{equation}\label{adjoint system of syst}
\left\{
\begin{array}{ll}
&-\pt \Phi = \deltax \Phi + A^*\Phi \mbox{ in } Q_{T}, \\
&\Phi(t,0)=0, \quad \Phi(t,1)=0 \mbox{ on } (0,T).
\end{array}
\right.
\end{equation}

The introduction of system (\ref{adjoint system of syst}) is of interest thanks to the following proposition, which gives a characterization of the null-controllability of system (\ref{syst3}) through an inequality on its adjoint system (\ref{adjoint system of syst}):
\begin{proposition}[Observability inequality]\label{observability inequality proposition}
Let be given $T> 0$. System (\ref{syst3}) is null-controllable on $(0,T)$  if and only if there exists $C>0$ such that for every $\Phi^T \in H^1_0(\Omega;\R^n)$ the solution $\Phi$ to the adjoint system (\ref{adjoint system of syst}) with $\Phi(T)=\Phi^T$ satisfies
\begin{equation}\label{observability inequality}
\begin{array}{ll}
{||\Phi(0)||}_{H^1_0(\Omega;\R^n)}^{2} \leq C \bigg( &
\ds \sum_{i=1}^{n_D} \int_{0}^{T} \int_{\omega_i} \abs{D_i^*\Phi(t,x)}^2 \, dx \, dt \\
&\ds +  \int_{0}^{T} \abs{B_{1}^*\px \Phi(t,0) }^2    dt +  \int_{0}^{T} \abs{B_{2}^*\px \Phi(t,1) }^2 dt
\bigg),
\end{array}
\end{equation}
\end{proposition}

For a proof see for instance \cite[Appendix]{FCGBdT}. Inequality (\ref{observability inequality}) is called \textit{observability inequality}.

\subsection{Proof of Theorem \ref{boundary Kalman condition theorem}}
The key point of the proof is to do an appropriate change of basis thanks to the hypothesis (\ref{boundary Kalman condition}). In this new basis, the matrix $A$ becomes a block upper triangular matrix $C$; and $B_1$, $B_2$ and $D$ become such that one control is acting on each diagonal block of $C$.
Note that this technique, firstly used in \cite{AKBDGB}, is very specific to the fact that the coefficients before the operator $-\deltax$ are the same on every single equation.
In a second time we will check that every diagonal block of $C$ satisfies the appropriate Kalman condition (boundary or distributed). And as a consequence, taking also advantage of the fact that the last block of $C$ is decoupled from the upper ones, we can start to control the last block in a time before $T$ and lead to zero at this time the components associated to this block; this allows us to iterate the process for the remaining blocks and finally lead every component to zero at time $T$.

\begin{proof}[Proof]
\textbf{Step 1}
Under the condition (\ref{boundary Kalman condition}) we start to construct a basis in which the matrices $A$, $D$ and $B_1$, $B_2$ has the desired structure. We have
\begin{lemma}\label{lemma for the construction of a basis}
Assume that condition (\ref{boundary Kalman condition}) holds. Then, there exists $r_D \in \{0,\ldots, n_D \}$, $D_{i_1}, \ldots, D_{i_{r_D}} \in \{D_k\}_{1 \leq k \leq n_D}$ and $s_1, \ldots, s_{r_D} \in \{1,\ldots, n \}$ such that for every $q \in \mathbb{N}^*$ there exists $r_B \in \{0,1, 2 \}$, $\notb{B_{j_1}}{}, \ldots, \notb{ B_{j_{r_B}} }{} \in \left\{ \notc{B_1}{q} ,\nota{ B_2}{q} \right\}$ and $\tilde{s}_1,\ldots \tilde{s}_{r_B} \in \{1,\ldots, nq \}$, such that
$$P_q=\left[ \ P^D_q \ \middle| \ P^B_q \ \right] \in \mat{nq}$$
is invertible, where we have denoted
$$
P^D_q=\notd{ \left[ \kalm{A}{D_{i_1}}{s_1} \middle| \cdots  \middle| \kalm{A}{D_{i_{r_D}}}{s_{r_D}}\right] }{q}
\mbox{ and }
P^B_q =\left[ \kalm{\calA_q}{\notb{B_{j_1}}{}}{\tilde{s}_1} \middle| \cdots \middle| \kalm{\calA_q}{\notb{B_{j_{r_B}}}{}}{\tilde{s}_{r_B}}  \right].
$$
Moreover for every $k \in \{1, \ldots, r_D\}$,
$$A^{s_k}D_{i_k} \in \vect \bigg[ \kalm{A}{D_{i_1}}{s_1} \bigg| \cdots  \bigg| \kalm{A}{D_{i_{k}}}{s_{k}}\bigg] .$$
\end{lemma}

\begin{proof}[Proof]
\textbf{Step 1}
We assume that $D \neq 0$ otherwise the result stated by Theorem \ref{boundary Kalman condition theorem} is already known (see \cite{AKBGBdT}). Thus there exists $D_{i_1} \in \{D_k\}_{1 \leq k \leq n_D}$ such that $D_{i_1} \neq 0$. We set
$$s_1= \rank{ \left( D_{i_1} , A D_{i_1} , \ldots , A^{n-1} D_{i_1} \right) }$$
so that $ \rank{ \left( D_{i_1} , A D_{i_1} , \ldots, A^{s_1-1} D_{i_1} \right) } =s_1$. If $s_1=n$ then the proof ends here by taking $P_q= \notd{ \left[ D_{i_1} | A D_{i_1} | \ldots | A^{s_1-1} D_{i_1} \right] }{q}$.
If $s_1 < n$ we check if there exists $D_{i_2} \in \{D_k\}_{1 \leq k \leq n_D} \backslash \, \{ D_{i_1} \}$ such that $(D_{i_1}, A D_{i_1}, \ldots , A^{s_1-1} D_{i_1}, D_{i_2})$ is linearly independent. If this is not the case we go to step 2. But if such a $D_{i_2}$ exists we set
$$s_2=\rank{ \left( D_{i_1} , A D_{i_1} , \ldots, A^{s_1-1} D_{i_1}, D_{i_2} , AD_{i_2}, \ldots , A^{n-1}D_{i_2}  \right) }  -s_1$$
so that  $\rank{ \left( D_{i_1} , A D_{i_1} , \ldots, A^{s_1-1} D_{i_1}, D_{i_2} , AD_{i_2}, \ldots , A^{s_2-1}D_{i_2}  \right) } =s_1+s_2$. If $s_1+s_2=n$ the proof ends. If $s_1 + s_2 <n$ then we continue the previous process. This stops when we have found a rank $r_D \in \{1,\ldots, n\}$, $i_1, \ldots, i_{r_D} \in \{1, \ldots, n_D\}$ and $s_1, \ldots, s_{r_D} \in \{1,\ldots, n \}$ such that
\begin{equation}\label{independent family}
\left( D_{i_1} ,A D_{i_1} , \ldots , A^{s_1-1} D_{i_1} , D_{i_2} , AD_{i_2}, \ldots , A^{s_2-1}D_{i_2}, \ldots , A^{s_{r_D}-1} D_{i_{r_D}} \right)
\end{equation}
is linearly independent and such that every element of $\{D_k\}_{1 \leq k \leq n_D} \backslash \, \{ D_{i_1}, \ldots , D_{i_{r_D}} \}$  belongs to the space spanned by the family (\ref{independent family}). As said before if $\sum_{k=1}^{r_D} s_k=n$ the proof ends (and let us remark that in this case system (\ref{syst3}) is null-controllable with distributed controls alone).
If this is not the case:

\textbf{Step 2}
Thanks to condition (\ref{boundary Kalman condition}) there exists $\notb{B_{j_1}}{} \in \left\{ \notc{B_1}{q} , \nota{B_2}{q} \right\}$ and $\hat{s} \in \{1,\ldots, nq\}$ such that
$$\left( \notd{ D_{i_1} ,A D_{i_1} , \ldots , A^{s_1-1} D_{i_1} , D_{i_2} , AD_{i_2}, \ldots , A^{s_2-1}D_{i_2}, \ldots , A^{s_{r_D}-1} D_{i_{r_D}} }{q} , \calA_q^{\hat{s}-1} \notb{B_{j_1} }{}  \right)$$
is linearly independent. One can check that this necessary implies that the familly
$$\left(  \notc{ D_{i_1} ,A D_{i_1} , \ldots , A^{s_1-1} D_{i_1} , D_{i_2} , AD_{i_2}, \ldots , A^{s_2-1}D_{i_2}, \ldots , A^{s_{r_D}-1} D_{i_{r_D}} }{q} , \notb{B_{j_1} }{}  \right)$$
is also linearly independent. We set
$$
\begin{array}{lr}
\tilde{s}_1= & \rank{ \left( \notd{ D_{i_1} , \ldots , A^{s_1-1} D_{i_1} , \ldots , D_{i_{r_D}}, \ldots, A^{s_{r_D}-1} D_{i_{r_D}} }{q} , \notb{B_{j_1} }{}, \ldots, \calA_q^{nq-1} \notb{B_{j_1} }{} \right) } \\
& \ds - \sum_{k=1}^{r_D} s_kq.
\end{array}$$
If $\tilde{s}_1+ \sum_{k=1}^{r_D} s_kq=nq$ we have done. If this is not the case, thanks to condition (\ref{boundary Kalman condition}) we can find $\notb{B_{j_2}}{} \in \left\{ \notc{B_1}{q}, \nota{B_2}{q} \right\} \backslash \left\{ \notb{B_{j_1}}{} \right\}$ such that
$$\left( \notd{ D_{i_1} , \ldots , A^{s_1-1} D_{i_1} , \ldots , D_{i_{r_D}}, \ldots, A^{s_{r_D}-1} D_{i_{r_D}} }{q} , \notb{B_{j_1} }{}, \ldots, \calA_q^{\tilde{s}_1-1} \notb{B_{j_1} }{}, \notb{B_{j_2} }{} \right)$$
is linearly independent. We iterate the same process and finally obtain the result.
\end{proof}

We apply Lemma \ref{lemma for the construction of a basis} and for the sake of simplicity of the notations we will treat one case, that is $r_B=2$ and $\notb{B_{j_1}}{}=\notc{B_1}{q}$, $\notb{B_{j_2}}{}=\nota{B_2}{q}$. For $q=1$ we obtain that $P_1 \in \mat{n}$ is invertible and
$$
P_1^{-1}AP_1=C=\left[\begin{array}{ccccc}
C_1 & \times & \cdots & \cdots & \times  \\
0 & \ddots & \ddots & & \vdots \\
\vdots & \ddots & \ddots & \ddots & \vdots \\
\vdots &  & \ddots & C_{r_D} & \times  \\
0 & \cdots &  \cdots & 0 & K
\end{array} \right]
\mbox{ where }
C_{i}=\left[\begin{array}{ccccc}
0 & 0 & \cdots & 0 &\times \\
1 & \ddots &   & \vdots & \vdots \\
0 &1 & \ddots & \vdots & \vdots \\
\vdots & \ddots & \ddots & 0 & \vdots \\
0 & \cdots & 0 & 1 & \times
\end{array}\right]
\in \mat{s_i}
$$
and $K \in \mat{ \tilde{s}_{B} }$ with $\tilde{s}_{B}= \sum_{k=1}^{2} \tilde{s}_k$.
Moreover for every $l \in \{1 , \ldots , r_D\}$ we have
$$P_1e_{S_l}=D_{i_l} \mbox{ and } P_1e_{\tilde{S}_1}=B_{1}, \quad P_2e_{\tilde{S}_2}=-B_{2}$$
where we denote $S_k =1+ \sum_{r=1}^{k-1} s_r$, $\tilde{S}_l = S_{r_D+1}+ \sum_{r=1}^{l-1} \tilde{s}_r$ and the vector $e_j$ denotes the real vector of $\R^n$ with $1$ on its $j$-th component and $0$ elsewhere.

Let us now remark that if the system
\begin{equation}\label{effective system}
\left\{
\begin{array}{ll}
&\pt z = \deltax z + Cz + e_{S_1}\hat{u}_{i_1}(t,x) 1_{\omega_{i_1}}(x) +\ldots+e_{S_{r_D}}\hat{u}_{i_{r_D}}(t,x) 1_{\omega_{i_{r_D}}}(x) \mbox{ in } Q_{T}, \\
&z(t,0)=e_{\tilde{S}_1}\hat{v}_{1}(t), \quad z(t,1)=-e_{\tilde{S}_{2}}\hat{v}_{2}(t) \mbox{ on } (0,T),
\end{array}
\right.
\end{equation}
is null-controllable, then system (\ref{syst3}) is also null-controllable by doing the change of variables $z=P_1^{-1}y$ and then taking for all $l \in \{1 , \ldots , n_D\}$ $u_{l}=\hat{u}_{i_k}$ if there exists $k \in \{1,\ldots, r_D \}$ such that $l=i_k$, $u_l=0$ otherwise, and taking for all $l \in \{1 , 2\}$ $v_{l}=\hat{v}_{l} $.
So let us now prove that the system (\ref{effective system}) is null-controllable:

\textbf{Step 2}
We rewrite the solution $z$ of system (\ref{effective system}) as follow:
$$z=\left[\begin{array}{c} z_1 \\ \vdots \\ z_{r_D} \\ z_B\end{array}\right]$$
where $z_B \in \R^{\tilde{s}_B}$ and $z_i \in \R^{s_i}$ for all $i \in \{1, \ldots, r_D\}$.

Now we look at the system satisfied by $z_B$ and observe that it is independent of $z_1, \ldots z_{r_D}$:
\begin{equation}\label{system satisfied by z_{r_D}}
\left\{
\begin{array}{ll}
&\pt z_B = \deltax z_B + Kz_B  \mbox{ in } Q_{T}, \\
&z_B(t,0)=e_{1}\hat{v}_{1}(t), \quad z_B(t,1)=-e_{\tilde{s}_{1}}\hat{v}_{2}(t)  \mbox{ on } (0,T).
\end{array}
\right.
\end{equation}
Assume for the moment that $K$ satisfies the boundary Kalman condition
\begin{equation}\label{boundary Kalman condition for the boundary block}
\rank{ \kalm{\calK_q}{ \left(\notc{e_{1}}{q} \middle| \nota{ -e_{\tilde{s}_1} }{q} \right) }{\tilde{s}_Bq} }= \tilde{s}_Bq, \quad \forall q \in \mathbb{N}^*,
\end{equation}
where we recall that
$$
\calK_q=\left[\begin{array}{ccccc}
K_1 & 0 & \cdots & \cdots & 0 \\
0 & K_2 & \ddots & & \vdots \\
\vdots & \ddots & \ddots & \ddots & \vdots \\
\vdots &  & \ddots & \ddots &0 \\
0 & \cdots &  \cdots & 0 & K_q
\end{array} \right]
\in \mat{\tilde{s}_Bq}
\mbox{ and }
K_k=-\lambda_k I +K
\in \mat{\tilde{s}_B}.
$$

Then, we deduce that the system (\ref{system satisfied by z_{r_D}}) is null-controllable. In particular, let a time $T_{B} \in (0,T)$ be given, then there exist controls $\hat{\hat{v}}_{1}, \hat{\hat{v}}_{2} \in L^2(0,T_{B})$ such that $z_B(T_{B})=0$ in $\Omega$. We choose
$$
\hat{v}_{1}(t)= \left\{ \begin{array}{ll} & \hat{\hat{v}}_{1}(t) \mbox{ if } t \in (0,T_B), \\ & 0 \mbox{ otherwise. } \end{array}\right.
\quad
\hat{v}_{2}(t)= \left\{ \begin{array}{ll} & \hat{\hat{v}}_{2}(t) \mbox{ if } t \in (0,T_B), \\ & 0 \mbox{ otherwise. } \end{array}\right.
$$
as controls, and one can see that $z_{B}(t)=0$ in $\Omega$ for all $t \geq T_{B}$.
As a consequence $\hat{z}$ defined by
$$\hat{z}=\left[\begin{array}{c}  z_1 \\ \vdots \\ z_{r_D} \end{array}\right]$$
satisfies
$$
\left\{
\begin{array}{ll}
&\pt \hat{z} = \deltax \hat{z} + \hat{C}\hat{z} + e_{S_1}\hat{u}_{i_1}(t,x) 1_{\omega_{i_1}}(x) +\ldots+e_{S_{r_D}}\hat{u}_{i_{r_D}}(t,x) 1_{\omega_{i_{r_D}}}(x) \mbox{ in } Q_{(T_{B}, T)}, \\
&\hat{z}(t,0)=0, \quad \hat{z}(t,1)=0 \mbox{ on } (T_{B},T),
\end{array}
\right.
$$
with
$$
\hat{C}=\left[\begin{array}{ccccc}
C_1 &  \times  & \cdots & \cdots &  \times  \\
0 & \ddots & \ddots & & \vdots \\
\vdots & \ddots & \ddots & \ddots & \vdots \\
\vdots &  & \ddots & \ddots &  \times  \\
0 & \cdots &  \cdots & 0 & C_{r_D}
\end{array} \right].
$$
And since for all $i \in \{1, \ldots, r_D \}$ the distributed Kalman condition $\rank{ \kalm{C_i}{e_i}{s_i} }=s_i$ is satisfied we can iterate this process and this will lead the result.

As a consequence it remains to prove that the condition (\ref{boundary Kalman condition for the boundary block}) is satisfied:

\textbf{Step 3}
In fact, condition (\ref{boundary Kalman condition for the boundary block}) holds if and only if for all $q \in \mathbb{N}^*$ the Hautus test holds (see Proposition \ref{prop Hautus test}):
\begin{equation}\label{Hautus test}
\ker \left( {\calK}_{q}^{*} - \theta I \right)
\, \cap \, \ker {  \left(\notc{e_{1}}{q} \middle| \nota{ -e_{\tilde{s}_1} }{q} \right) }^{*}
= \{0 \}, \quad \forall \theta \in \C.
\end{equation}

To prove (\ref{Hautus test}) we will use condition (\ref{boundary Kalman condition}), and to this aim let us then first reformulate (\ref{Hautus test}) in terms of the original data of the problem, that is $A$, $D$, $B_1$ and $B_2$. This is done through the following lemma:

\begin{lemma}\label{equivalences}
Assume that condition (\ref{boundary Kalman condition}) holds. We have the following equivalences:
\begin{enumerate}
\item\label{item 1} For all $q \in \mathbb{N}^*$ the Hautus test (\ref{Hautus test}) holds.

\item\label{item 2} For all $q \in \mathbb{N}^*$, for all $\theta \in \C$, for all $s \in \mathbb{N}^*$ and for all $V^1, \ldots , V^s \in \R^{\tilde{s}_Bq}$ linearly independent vectors of $\ker \left( \calK_q^* - \theta I \right)$, the set
$$\left\{ \quad {  \left(\notc{e_{1}}{q} \middle| \nota{ -e_{\tilde{s}_1} }{q} \right) }^* V^k \quad \right\}_{1 \leq k \leq s}$$
is linearly independent in $\R^{2q}$.

\item\label{item 3} For all $q \in \mathbb{N}^*$, for all $\theta \in \C$, for all $s \in \mathbb{N}^*$ and for all $W^1, \ldots , W^s \in \R^{nq}$ linearly independent vectors of $\ker \left( \calA_q^* - \theta I \right) \, \cap \, \ker \notd{D_{i_1} | \cdots | D_{i_{r_D}} }{q}^*$, the set
$$
{\left\{ \quad { \left( \notb{B_{j_1}}{} \middle| \notb{B_{j_2}}{} \right) }^{*} W^k \quad \right\} }_{1 \leq k \leq s}
$$
is linearly independent in $\R^{2q}$.
\end{enumerate}

And those conditions are implied by the following one: for all $q \in \mathbb{N}^*$ we have
\begin{equation}\label{an equivalent to the Hautus test}
\ker \left( {\calA}_{q}^{*} - \theta I \right)
\, \cap \, \ker \notd{D_{i_1} | \cdots | D_{i_{r_D}} }{q}^*
\, \cap \, \ker { \left( \notb{B_{j_1}}{} \middle| \notb{B_{j_2}}{} \right) }^{*}
= \{0 \}, \quad \forall \theta \in \C.
\end{equation}

\end{lemma}

\begin{proof}[Proof]
One can see that item \ref{item 1} is equivalent to item \ref{item 2} (see for instance \cite[Proposition 3.1]{AKBGBdT}) and that condition (\ref{an equivalent to the Hautus test}) implies item \ref{item 3}. Thus let us prove that item \ref{item 2} and item \ref{item 3} are equivalent. Let us fix $q \in \mathbb{N}^*$ and $\theta \in \C$. We define a bijective linear map $\Phi$ by:
$$
\begin{array}{ccccc}
\Phi & : & \ker \left( \calK_q^*-\theta I \right) & \longrightarrow & \ker \left( \calA_q^*- \theta I \right) \, \cap \, E \\
       &    & V=\left[\begin{array}{c} V_1 \\ \vdots \\ V_q \end{array}\right]& \longmapsto & \left[\begin{array}{c} \tilde{\Phi}(V_1) \\ \vdots \\ \tilde{\Phi}(V_q) \end{array}\right]
\end{array}
$$
where
$$\tilde{\Phi}(V_l)={(P_1^*)}^{-1} \left[\begin{array}{c} 0 \\ V_l \end{array}\right] \in \R^n,$$
and
$$E =\left\{ 
W = \left[\begin{array}{c} W_1 \\ \vdots \\ W_q \end{array}\right] \in \R^{nq} 
\mbox{ such that }
\forall l \in \{1, \ldots, q \}, \quad
W_l=(P_1^*)^{-1} \left[ \begin{array}{c} 0 \\ \times \end{array} \right]
\right\}.$$
One can check that in fact $E=\ker (P^D_q)^*$ and thus $\ker \left( \calA_q^*- \theta I \right) \, \cap \, E=\ker \left( \calA_q^* - \theta I \right) \, \cap \, \ker \notd{D_{i_1} | \cdots | D_{i_{r_D}} }{q}^*$.
Moreover, for all $s \in \mathbb{N}^*$ and all $\{\alpha_k\}_{1 \leq k \leq s} \subset \R$ we have
$$
\begin{array}{ll}
\ds \sum_{k=1}^{s} \alpha_k { \left( \notb{B_{j_1}}{} \middle| \notb{B_{j_2}}{} \right) }^* \Phi(V^k)
&\ds =\sum_{k=1}^{s} \alpha_k \left( \sum_{l=1}^{q} {\notc{B_{1} \middle| {(-1)}^{l} B_{2} }{}}^* \Phi_l(V^k_l) \right) \\
&\ds =\sum_{k=1}^{s} \alpha_k \left( \sum_{l=1}^{q} \left(e_{\tilde{S}_1} \middle| {(-1)}^{l+1} e_{\tilde{S}_2} \right)^* P_1^* \Phi_l(V^k_l) \right) \\
&\ds =\sum_{k=1}^{s} \alpha_k  \left( \sum_{l=1}^{q} \left(e_{\tilde{S}_1} \middle| {(-1)}^{l+1} e_{\tilde{S}_2} \right)^* \left[ \begin{array}{c} 0 \\ V^k_l \end{array} \right] \right) \\
&\ds =\sum_{k=1}^{s} \alpha_k \left( \sum_{l=1}^{q} \left(e_{1} \middle| {(-1)}^{l+1}e_{\tilde{s}_1} \right)^* V^k_l \right) \\
&\ds =\sum_{k=1}^{s} \alpha_k  {  \left(\notc{e_{1}}{q} \middle| \nota{ -e_{\tilde{s}_1} }{q} \right) }^* V^k.
\end{array}
$$
Combining those two facts the claim is proved.

\end{proof}

As a consequence of Lemma \ref{equivalences} it is sufficient to prove that (\ref{an equivalent to the Hautus test}) is true, and in fact it is a consequence of the Hautus test, and hypothesis (\ref{boundary Kalman condition}) :
\begin{lemma}
Assume that condition (\ref{boundary Kalman condition}) holds. Then, for all $q \in \mathbb{N}^*$ we have
\begin{equation}\label{almost Hautus test}
\ker \left( {\calA}_{q}^{*} - \theta I \right)
\, \cap \, \ker \notd{D_{i_1} | \cdots | D_{i_{r_D}} }{q}^*
\, \cap \, \ker { \left( \notb{B_{j_1}}{} \middle| \notb{B_{j_2}}{} \right) }^{*}
= \{0 \}, \quad \forall \theta \in \C.
\end{equation}
\end{lemma}

\begin{proof}[Proof]
(\ref{almost Hautus test}) can be rewritten as
$$
\ker \left( {\calA}_{q}^{*} - \theta I \right)
\, \cap \, \ker { \left( \notb{B_{j_1}}{} \middle| \notb{B_{j_2}}{} \middle|
\notd{D_{i_1} | \cdots | D_{i_{r_D}} }{q} \right) }^{*}
= \{0 \}, \quad \forall \theta \in \C.
$$
which is equivalent to (by the Hautus test)
$$\rank{ \kalm{\calA_q}{  \left( \notb{B_{j_1}}{} \middle| \notb{B_{j_2}}{} \middle|
\notd{D_{i_1} | \cdots | D_{i_{r_D}} }{q} \right) }{nq} }=nq,$$
and this last formulation is also equivalent to
\begin{equation}\label{last reformulation}
\rank{ \left[ \kalm{\calA_q}{ \left( \notb{B_{j_1}}{} \middle| \notb{B_{j_2}}{} \right) }{nq} \middle| \notd{\kalm{A}{ \left(D_{i_1} | \cdots | D_{i_{r_D}} \right) }{n} }{q} \right] }=nq,
\end{equation}
in the same way as condition (\ref{boundary Kalman condition}) is equivalent to condition (\ref{boundary Kalman condition 2}) (see Remark \ref{another characterization of the boundary Kalman condition}, item 2).
Now thanks to Lemma \ref{lemma for the construction of a basis} we can see that (\ref{last reformulation}) holds (observe that we have more powers of $\calA_q$ and $A$ in (\ref{last reformulation})).

\end{proof}

Let us now prove the necessary part of Theorem \ref{boundary Kalman condition theorem}:

\textbf{Step 4} Suppose that there exists $q_0 \in \mathbb{N}^*$ such that
$$\rank{ \left[ \ \kalm{\calA_{q_0} }{  \left( \notc{B_1}{q_0} \middle| \nota{B_2}{q_0} \right)  }{nq_0} \ \middle| \ \notd{ \kalm{A}{D}{n}}{q_0} \ \right] } <nq_0.$$
Thanks to the other characterization of condition (\ref{boundary Kalman condition}) (see Remark \ref{another characterization of the boundary Kalman condition}, item 2) this means we have
\begin{equation}\label{eq000}
\rank{ \kalm{\calA_{q_0} }{ \left( \notc{B_1}{q_0} \middle| \nota{B_2}{q_0}  \middle| \notd{D}{q_0} \right) }{nq_0} }< nq_0.
\end{equation}
Thus, there exists $\Psi^T \in \R^{nq_0}$ with $\Psi^T \neq 0$ such that $\Psi(t)=e^{\calA_{q_0}^*(T-t)}\Psi^T$ satisfies (see Proposition \ref{prop Hautus test})
\begin{equation}\label{Kalman for ODE}
{ \left( \notc{B_1}{q_0} \middle| \nota{B_2}{q_0}  \middle| \notd{D}{q_0} \right) }^{*} \Psi(t) = 0, \quad \forall t \in [0,T].
\end{equation}
Let us write $\Psi(t)$ as follow:
$$\begin{array}{ll}
\Psi(t)=e^{\calA_{q_0}^*(T-t)}\Psi^T & =
\left[\begin{array}{ccccc}
e^{A_1^*(T-t)} & 0 & \cdots & \cdots & 0 \\
0 & e^{A_2^*(T-t)} & \ddots & & \vdots \\
\vdots & \ddots & \ddots & \ddots & \vdots \\
\vdots &  & \ddots & \ddots &0 \\
0 & \cdots &  \cdots & 0 & e^{A_{q_0}^*(T-t)}
\end{array} \right]
\left[\begin{array}{c}
\Psi^T_1 \\ \Psi^T_2 \\ \vdots \\  \vdots \\ \Psi^T_{q_0}
\end{array}\right] \\
& =\left[\begin{array}{c}
e^{(-\lambda_1 I + A^*)(T-t)} \Psi^T_1 \\ e^{(-\lambda_2 I + A^*)(T-t)}  \Psi^T_2 \\ \vdots \\  \vdots \\  e^{(-\lambda_{q_0} I + A^*)(T-t)}  \Psi^T_{q_0}
\end{array}\right]
=\left[\begin{array}{c}
\Psi_1(t) \\ \Psi_2(t) \\ \vdots \\  \vdots \\  \Psi_{q_0}(t)
\end{array}\right], \quad \forall t \in [0,T].
\end{array}
$$
Thus (\ref{Kalman for ODE}) gives
$$
\left[\begin{array}{cccccc}
B_1^* & B_1^* & \cdots & \cdots & \cdots  & B_1^* \\
-B_2^* & B_2^* & \cdots & \cdots & \cdots  & {(-1)}^{q_0}B_2^* \\
D^* & 0 & \cdots& \cdots & \cdots  &0 \\
0 & D^* & \ddots &  &  & \vdots \\
\vdots & \ddots & \ddots&  & \ddots & \vdots \\
\vdots &   &  \ddots & & \ddots & 0 \\
0 & \cdots &  \cdots & & 0 & D^* 
\end{array} \right]
\left[\begin{array}{c}
\Psi_1(t) \\ \Psi_2(t) \\ \vdots \\ \vdots \\ \vdots \\ \vdots \\ \Psi_{q_0}(t)
\end{array}\right]
=0, \quad \forall t \in  [0,T],
$$
i.e.
\begin{equation}\label{unique continuation hypothesis}
\forall t \in [0,T], \quad \left\{\begin{array}{l}
\ds \sum_{k=1}^{q_0} B_1^* \Psi_k(t)=0, \\
\ds \sum_{k=1}^{q_0} {(-1)}^{k} B_2^* \Psi_k(t)=0, \\
\ds D^*\Psi_k(t)=0, \quad \forall k \in \{1, \ldots , q_0\}.
\end{array}\right.
\end{equation}
Let us now prove that the observability inequality (\ref{observability inequality}) fails. To this aim we define $\Phi^T \in H^{1}_{0}(\Omega;\R^n )$ by
$$\Phi^T=\sum_{k=1}^{q_0} {\lambda_k}^{-1/2} \Psi^T_k  \phi_k.$$
and let $\Phi$ be the solution to
$$
\left\{
\begin{array}{ll}
&-\pt \Phi = \deltax \Phi + A^*\Phi \mbox{ in } Q_{T}, \\
&\Phi=0 \mbox{ on } \Sigma_{T}, \\
&\Phi(T)=\Phi^T \mbox{ in } \Omega,
\end{array}
\right.
$$
i.e.
\begin{multline}\label{explicite formula}
\Phi(t) \ds =\sum_{k=1}^{q_0} e^{(-\lambda_k I + A^*)(T-t)} 
\left[\begin{array}{c}
{\left\langle \Phi^T_1 , \phi_k \right\rangle}_{L^2} \\
\vdots \\
{\left\langle \Phi^T_n , \phi_k \right\rangle}_{L^2}
\end{array}\right]
\phi_k \\
=\sum_{k=1}^{q_0} {\lambda_k}^{-1/2} e^{(-\lambda_k I + A^*)(T-t)}  \Psi^T_k \phi_k=\sum_{k=1}^{q_0} {\lambda_k}^{-1/2} \Psi_k(t) \phi_k.
\end{multline}
From (\ref{unique continuation hypothesis}) we obtain
$$\forall t \in (0,T), \forall i \in \{1,\ldots,n_D\}, \quad D_i^*\Phi(t) =\sum_{k=1}^{q_0} {\lambda_k}^{-1/2}  D_i^* \Psi_k(t) \phi_k 
=0 \mbox{ in } \Omega,$$
and since the space dimension is $N=1$ we also obtain
$$\forall t \in (0,T), \quad B_1^*\px \Phi(t,0)=  B_1^*\sum_{k=1}^{q_0}  \Psi_k(t)  \underbrace{ {\lambda_k}^{-1/2}  \px {\phi}_{k}(0) }_{=1}  =0,$$
and
$$\forall t \in (0,T), \quad B_2^*\px \Phi(t,1)= B_2^* \sum_{k=1}^{q_0}  \Psi_k(t)  \underbrace{ {\lambda_k}^{-1/2}  \px {\phi}_{k}(1)  }_{= {(-1)}^{k} }=0.$$
Finally let us remark that $\Phi(0) \neq 0$ since $\Psi^T \neq 0$ (see (\ref{explicite formula})). As a consequence the observability inequality (\ref{observability inequality}) fails and so does the null-controllability of (\ref{syst3}).
\end{proof}

\section{The Carleman estimate}

Recall that all along this section, no boundary controls are considered and the space dimension $N$ is arbitrary.
The aim is to prove Theorem \ref{t-dep} and as said before this latter is a consequence of a Carleman estimate.

\subsection{A Carleman estimate for cascade systems}

Let us introduce the framework in which the Carleman estimate will be established. We consider this time the following $n\times n$ parabolic system:
\begin{equation}\label{cascade system}
\left\{ 
\begin{array}{ll}
&\pt y = \Delta y + Cy + e_{S_1} u_1(t,x) 1_{\omega_1}(x) +\ldots+ e_{S_r}u_r(t,x) 1_{\omega_r}(x) \mbox{ in } Q_{T}, \\
&y=0 \mbox{ on } \Sigma_{T}, 
\end{array}
\right.
\end{equation}
where $r \in \{1,\ldots, n\}$ is the number of controls, $u_1 \in L^2(Q_{T} ),\ldots, u_r \in L^2(Q_{T})$ are the controls and where $C=C(t,x) \in L^{\infty}\left(Q_{T};\mat{n} \right)$ is a matrix with the following block cascade type structure:
\begin{equation}\label{cascade form}
C=\left[\begin{array}{ccccc}
C_{11} & \times & \times & \cdots & \times \\
0 & C_{22} & \times & \cdots & \times \\
\vdots & \ddots & \ddots & \ddots & \vdots \\
\vdots &  & \ddots & \ddots & \times \\
0 & \cdots &  \cdots & 0 & C_{rr}
\end{array} \right]
\mbox{ with }
C_{jj}=\left[\begin{array}{ccccc}
c_{11}^{j} & c_{12}^{j} & c_{13}^{j} & \cdots & c_{1s_j}^{j} \\
c_{21}^{j} & c_{22}^{j} & c_{23}^{j} & \cdots & c_{2s_j}^{j} \\
0 & c_{32}^{j} & c_{33}^{j} & \cdots & c_{3s_j}^{j} \\
\vdots & \ddots & \ddots & \ddots & \vdots \\
0 & \cdots & 0 & c_{s_j s_{j-1}}^{j} & c_{s_j s_j}^{j}
\end{array}\right],
\end{equation}
where $s_j \in \mathbb{N}$ is the size of the bloc $C_{jj}$ (in particular we have $\sum_{j=1}^{r} s_j=n$) and where $S_1,\ldots, S_{r+1}$ are such that one control is exerted on each first equation of a block, this reads $S_i=1+\sum_{j=1}^{i-1} s_j$ for all $i \in \{1, \ldots, r+1\}$.

Those notations in mind, we have
\begin{theorem}\label{Carleman estimate for coupled linear parabolic systems of cascade type theorem}
Assume that for every $j\in \{1, \ldots, r\}$ there exists a nonempty open subset $\tilde{\omega}_j \subset \omega_j$ and $\underline{c}^j>0$ such that for every $i \in \{1, \ldots, s_j\}$ we have
\begin{equation}\label{non-zero hypothesis}
c^{j}_{i+1,i} \geq \underline{c}^j \mbox{ or } -c^{j}_{i+1,i} \geq \underline{c}^j \mbox{ on } (0,T) \times \tilde{\omega}_j.
\end{equation}

Then, there exist functions $\beta_1,\ldots,\beta_{r} \in C^2(\overline{\Omega})$ such that $0<\beta_1<\ldots<\beta_{r}$, there exists $C>0$, $s_0>0$ and $l_0>0$ such that, for every $\Phi^T \in L^2(\Omega;\R^n )$, the solution $\Phi$ to the system
\begin{equation}\label{adjoint system}
\left\{ \begin{array}{ll}
&-\pt \Phi = \Delta \Phi +C^*\Phi \mbox{ in } Q_{T}, \\
&\Phi=0 \mbox{ on } \Sigma_{T}, \\
&\Phi(T)=\Phi^T \mbox{ in } \Omega,
\end{array} \right.
\end{equation}
satisfies
$$\sum_{j=1}^{r} \sum_{i=S_{j}}^{S_{j+1}-1} \calJ_j(3(S_{j+1}-i),\Phi_i) \leq C \sum_{j=1}^{r}\iint_{(0,T) \times \omega_j} {(s\varphi)}^{ l_0 } e^{-2s \eta_j} {|\Phi_{S_j}|}^{2},$$
for all $s \geq s_0$. Here we have denoted
\begin{equation}\label{definition of calJ}
\calJ_j(d,q)= \iint_{Q_{T}} {(s\varphi)}^{d-2}e^{-2s \eta_j}  \abs{\gradx q}^{2} + \iint_{Q_{T}} {(s \varphi)}^d e^{-2s\eta_j} \abs{q}^{2},
\end{equation}
and $\varphi(t)={\left(t(T-t)\right)}^{-1}$, $\eta_j(t,x)=\beta_j(x)\varphi(t)$ for $j \in \{1,\ldots, r\}$.

\end{theorem}

\begin{proof}[Proof]

We adapt the proof of \cite[Theorem 1.1]{GBdT}, but we consider different weight functions on each block and we use the fact that we still can choose such functions in an ordered way.
First, let us rewrite system (\ref{adjoint system}) on each block as follows:
\begin{equation}\label{adjoint system rewrote}
\forall j \in \{1, \ldots,r \}, \quad \left\{
\begin{array}{ll}
&\ds -\pt \Phi_{i} = \Delta \Phi_{i} +\sum_{k=1}^{i} {c}_{ki}\Phi_k + {c}_{i+1,i}\Phi_{i+1}  \mbox{ in } Q_{T}, \quad \forall i \in \{S_j, \ldots S_{j+1}-2\}\\
& \\
&\ds -\pt \Phi_{ S_{j+1}-1 } = \Delta \Phi_{ S_{j+1}-1 } + \sum_{k=1}^{S_{j+1}-1 } {c}_{k,S_{j+1}-1 }\Phi_k \mbox{ in } Q_{T}, \\
& \\
&\ds \Phi_{i} =0 \mbox{ on } \Sigma_{T}, \quad \forall i \in \{S_j, \ldots, S_{j+1}-1\}.
\end{array}
\right.
\end{equation}
where we rewrote for convenience $C={(c_{ij})}_{1 \leq i,j \leq n}$. And let us recall the following Carleman estimate for one single parabolic equation (see \cite[Lemma 2.3]{IY}):
\begin{lemma}\label{Carleman estimate for one equation}
Let $\omega \subset \Omega$ be a non-empty open subset. For every $\underline{\beta}>0$, there exists a function $\beta \in C^2(\overline{\Omega})$ such that $\beta>\underline{\beta}$ and such that, for every $d \in \mathbb{R}$, there exist $C>0$, $s_0>0$ such that, for every  $\psi^T \in L^2(\Omega)$ and every $f \in L^2(Q_{T})$, the solution $\psi$ to
$$
\left\{ \begin{array}{ll}
&-\pt \psi =\Delta \psi + f(t,x) \mbox{ in } Q_{T}, \\
&\psi=0 \mbox{ on } \Sigma_{T}, \\
&\psi(T)=\psi^T \mbox{ in } \Omega,
\end{array}
\right.
$$
satisfies
\begin{multline*}
\iint_{Q_{T}} (s\varphi)^{d-2}e^{-2s \eta}  \abs{\gradx \psi}^{2} + \iint_{Q_{T}} (s\varphi)^de^{-2s\eta} \abs{\psi}^{2} \\
\leq C \left( \iint_{(0,T) \times \omega}(s\varphi)^d e^{-2s\eta} \abs{\psi}^{2} + \iint_{Q_{T}} (s\varphi)^{d-3} e^{-2s \eta} \abs{f}^{2} \right)
\end{multline*}
for all $s \geq s_0$. Where $\eta(t,x)=\beta(x)\varphi(t)$.
\end{lemma}

To start the proof of Theorem \ref{Carleman estimate for coupled linear parabolic systems of cascade type theorem}, let be given $\tilde{\tilde{\omega}}_1 \subset \subset \tilde{\omega}_1, \ldots , \tilde{\tilde{\omega}}_p \subset \subset \tilde{\omega}_p$. For every $j \in \{1,\ldots,r\}$ we apply Lemma \ref{Carleman estimate for one equation} with $\omega=\tilde{\tilde{\omega}}_j$ which enables us to construct functions $\beta_1,\ldots,\beta_{r} \in C^2(\overline{\Omega})$ such that $0<\beta_1<\ldots<\beta_{r}$ and such that each function $\Phi_i$, $S_j \leq i \leq S_{j+1}-1$, satisfies, due to the particular structure of $C^*$ (see (\ref{adjoint system rewrote})):
\begin{multline*}
\forall i \neq S_{j+1}-1, \quad
\calJ_j(3(S_{j+1}-i),\Phi_i)
\leq  \frac{1}{2} C_1 \bigg( \calL_j \left(\tilde{\tilde{\omega}}_j;3(S_{j+1}-i),\Phi_i \right) \\
+\calJ_j(3(S_{j+1}-1-i),\Phi_{i+1}) )
 +\sum_{k=1}^{i} \calJ_j(3(S_{j+1}-1-i), \Phi_k) \bigg),
\end{multline*}
and
$$\calJ_j(3,\Phi_{S_{j+1}-1})
\leq \frac{1}{2} C_1 \left( \calL_j \left(\tilde{\tilde{\omega}}_j;3(S_{j+1}-i),\Phi_{S_{j+1}-1} \right) 
+\sum_{k=1}^{S_{j+1}-1} \calJ_j(0, \Phi_k) \right),$$
for $s$ large enough, where here and in what follows we denote
$$\calL_k(\omega;d,q)= \iint_{(0,T) \times \omega}(s\varphi)^d e^{-2s\eta_k} {|q|}^{2}.$$
Summing over $i$ this gives
$$
\begin{array}{ll}
\ds \sum_{i=S_{j}}^{S_{j+1}-1} {C_1}^{i-S_j}  \calJ_j(3(S_{j+1}-i),\Phi_i)
\leq & \ds \sum_{i=S_{j}}^{S_{j+1}-1} \frac{1}{2} {C_1}^{i+1-S_j}\bigg(  \calL_j\left(\tilde{\tilde{\omega}}_j;3(S_{j+1}-i),\Phi_i\right) \\
& \ds + \sum_{k=1}^{i} \calJ_j(3(S_{j+1}-1-i),\Phi_k)  \bigg).
\end{array}
$$
Summing over $j$ this leads to
\begin{equation}\label{tired}
\begin{array}{ll}
\ds \sum_{j=1}^{r}\sum_{i=S_{j}}^{S_{j+1}-1}  \calJ_j(3(S_{j+1}-i),\Phi_i)
\leq & \ds {C_2} \bigg( \sum_{j=1}^{r} \sum_{i=S_{j}}^{S_{j+1}-1} \calL_j\left(\tilde{\tilde{\omega}}_j;3(S_{j+1}-i),\Phi_i\right) \\
& \ds +\sum_{j=1}^{r} \sum_{i=S_{j}}^{S_{j+1}-1} \sum_{k=1}^{i} \calJ_j(3(S_{j+1}-1-i),\Phi_k)  \bigg).
\end{array}
\end{equation}
Now let us remark that the last term in the right hand-side of (\ref{tired}) can be seen as the sum of two terms:
$$
\begin{array}{ll}
\ds \sum_{j=1}^{r} \sum_{i=S_{j}}^{S_{j+1}-1} \sum_{k=1}^{i} \calJ_j(3(S_{j+1}-1-i),\Phi_k) 
= &\ds \sum_{j=1}^{r} \sum_{i=S_{j}}^{S_{j+1}-1} \calJ_j(3(S_{j+1}-1-i),\Phi_i) \\
&\ds +\sum_{j=1}^{r} \sum_{i=S_{j}}^{S_{j+1}-1} \sum_{k=1}^{i-1} \calJ_j(3(S_{j+1}-1-i),\Phi_k),
\end{array}
$$
and for $s$ large enough the first one can be absorbed by the left hand-side of (\ref{tired}), because $S_{j+1}-1-i<S_{j+1}-i$, and the second one can also be absorbed by the left hand-side of (\ref{tired}) since the functions $\beta_j$ have been constructed such that
$$0<\beta_1 < \beta_2 < \ldots < \beta_r,$$
so that, for every $d_1,d_2 \geq 0$, there exists $C>0$ such that
$$(s \varphi)^{d_1}e^{-2s\eta_j} \leq C (s \varphi)^{d_2}e^{-2s\eta_{j-1}}, \quad \forall  j \in \{2,\ldots, r\}$$ 
for $s$ large enough. Thus we finally obtain
\begin{equation}\label{after carleman}
\ds \sum_{j=1}^{r}\sum_{i=S_{j}}^{S_{j+1}-1} \calJ_j(3(S_{j+1}-i),\Phi_i)
\ds \leq {C_3} \left( \sum_{j=1}^{r}\sum_{i=S_{j}}^{S_{j+1}-1}\calL_j \left(\tilde{\tilde{\omega}}_j;3(S_{j+1}-i),\Phi_i \right) \right)
\end{equation}
for $s$ large enough. To pursue we use the following lemma (see \cite[Section 4]{GBdT}):

\begin{lemma}\label{bootstrap}
Assume that (\ref{non-zero hypothesis}) holds. Then, for every $\epsilon>0$, $j \in \{1,\ldots r \}$, $i \in \{S_j+1,\ldots, S_{j+1}-1\}$ and $l \in \mathbb{N}$ and every open sets $\mathcal{O}_0, \mathcal{O}_1$ such that $\tilde{\tilde{\omega}}_j \subset \mathcal{O}_1 \subset \subset \mathcal{O}_0 \subset \tilde{\omega}_j$, there exist $C>0$, $s_0>0$ and $l_{1}(j,l),\ldots,l_{i-1}(j,l) \in \mathbb{N}$ such that the solution $\Phi$ to (\ref{adjoint system}) satisfies
$$
\begin{array}{ll}
\forall i \neq S_{j+1}-1, \quad
\calL_j(\mathcal{O}_1;l,\Phi_i) \leq & \epsilon \bigg( \calJ_j(3(S_{j+1}-i), \Phi_i)
+ \calJ_j(3(S_{j+1}-1-i), \Phi_{i+1}) \bigg) \\
&\ds +C \sum_{k=1}^{i-1} \calL_j(\mathcal{O}_0;l_k(j,l) , \Phi_k).
\end{array}
$$
and
$$\calL_j(\mathcal{O}_1;l,\Phi_{S_{j+1}-1})
\leq \epsilon \calJ_j(3, \Phi_{S_{j+1}-1})
+ C \sum_{k=1}^{S_{j+1}-2} \calL_j(\mathcal{O}_0;l_k(j,l) , \Phi_k).$$
for all $s \geq s_0$.
\end{lemma}

For each $j \in \{1, \ldots, r \}$ let be given $\tilde{\tilde{\omega}}_j \subset \subset \mathcal{O}_{j,1} \subset \subset ... \subset \subset \mathcal{O}_{j,s_j-1} \subset \subset \tilde{\omega_j}$.
Applying Lemma \ref{bootstrap} to $i=S_{j+1}-1$, $\mathcal{O}_1=\tilde{\tilde{\omega}}_j$, $\mathcal{O}_0=\mathcal{O}_{j,1}$, $l=3(S_{j+1}-i)=l^1_{j}$ and $\epsilon=\frac{1}{2 C_3 }$, we have
$$\calL_j(\tilde{\tilde{\omega}}_j;l^1_{j},\Phi_{S_{j+1}-1}) \leq \frac{1}{2 C_3 } \calJ_j(3, \Phi_{S_{j+1}-1}) + C_4 \sum_{k=1}^{S_{j+1}-2} \calL_j(\mathcal{O}_{j,1};l_k(j,l^1_{j}) , \Phi_k).$$
Back to (\ref{after carleman}) we obtain
$$\sum_{j=1}^{r} \sum_{i=S_{j}}^{S_{j+1}-1} \calJ_j(3(S_{j+1}-i),\Phi_i)
\leq C_5 \left( \sum_{j=1}^{r} \sum_{k=1}^{S_{j+1}-2} \calL_j(\mathcal{O}_{j,1};\max\left\{3(S_{j+1}-k),l_k(j,l^1_{j})\right\} , \Phi_k) \right).$$
Applying now Lemma \ref{bootstrap} to $i=S_{j+1}-2$, $\mathcal{O}_1=\mathcal{O}_{j,1}$, $\mathcal{O}_0=\mathcal{O}_{j,2}$, $l=\max\left\{6,l_{S_{j+1}-2}(j,l^1_{j})\right\}=l^2_{j}$ and $\epsilon=\frac{1}{2 C_5 }$, we obtain
\begin{multline*}
\sum_{j=1}^{r} \sum_{i=S_{j}}^{S_{j+1}-1} \calJ_j(3(S_{j+1}-i),\Phi_i) \\
\leq C_6 \left( \sum_{j=1}^{r} \sum_{k=1}^{S_{j+1}-3} \calL_j(\mathcal{O}_{j,2};\max\left\{ \max\left\{3(S_{j+1}-k),l_k(j,l^1_{j})\right\} ,l_k(j,l^2_{j})\right\} , \Phi_k) \right).
\end{multline*}
Iterating the process leads to the estimate
\begin{multline*}
\sum_{j=1}^{r} \sum_{i=S_{j}}^{S_{j+1}-1} \calJ_j(3(S_{j+1}-i),\Phi_i) \\
\leq C_7 \left( \sum_{j=1}^{r} \sum_{k=1}^{S_j} \calL_j(\mathcal{O}_{j,s_j-1};\max\left\{3(S_{j+1}-k), l_k(j,l^{1}_{j}), \ldots, l_k(j,l^{s_j-1}_{j}) \right\} , \Phi_k) \right),
\end{multline*}
which can be rewriten as follow by separating the term $k=S_j$
\begin{multline}\label{last argue}
\sum_{j=1}^{r}\sum_{i=S_{j}}^{S_{j+1}-1} \calJ_j(3(S_{j+1}-i),\Phi_i)
\leq C_7 \bigg( \sum_{j=1}^{r}\calL_j(\mathcal{O}_{j,s_j-1}; \max\left\{3s_j, l_{S_j}(j,l^{1}_{j}), \ldots, l_{S_j}(j,l^{s_j-1}_{j}) \right\} , \Phi_{S_j}) \\
+ \sum_{j=1}^{r}\sum_{k=1}^{S_j-1} \calL_j(\mathcal{O}_{j,s_j-1}; \max\left\{3(S_{j+1}-k), l_k(j,l^{1}_{j}), \ldots, l_k(j,l^{s_j-1}_{j}) \right\} , \Phi_k) \bigg).
\end{multline}
Let us now recall that we have chosen $\beta_1 < \beta_2 < \ldots < \beta_r$ so that the last term in the right-hand side of (\ref{last argue}) can be absorbed, for $s$ large enough, by the term of the left-hand side of (\ref{last argue}), no matter what the power of $s$ in those terms are.

\end{proof}

\subsection{Proof of Theorem \ref{t-dep}}

All the work is based on to the previous Carleman estimate.
Indeed, following the ideas of \cite{AKBDGB2} we can construct a change of basis thanks to condition (\ref{kalm for t-dep}) which leads to a cascade system (see \cite[Lemma 4.1]{AKBDGB2}) with possibly controls acting on different subdomains.
Applying the previous Carleman estimate we deduce the result.

\section{Further results and comments}

\begin{enumerate}

\item
Untill now we looked at the null-controllability properties for systems where the coefficients in front of the operator $-\Delta$ were the same on every equation but we can also consider systems where those coefficients are different; let us consider
\begin{equation}\label{system with diffusion matrix}
\left\{
\begin{array}{ll}
&\pt y = J \Delta y + Ay + D_1 u_1(t,x) 1_{\omega_1}(x) +\ldots+ D_{n_D} u_{n_D}(t,x) 1_{\omega_{n_D}}(x) \mbox{ in } Q_{T}, \\
&y=0 \mbox{ on } \Sigma_{T},
\end{array}
\right.
\end{equation}
with $J \in \mat{n}$ such that $J$ is diagonalizable with positive eigenvalues. Then there exists  also a Kalman rank condition, which is a simple extension of the one proved in \cite{AKBD}:

\begin{theorem}\label{Kalman condition with diffusion matrix theorem}
Under the previous assumption on $J$, system (\ref{system with diffusion matrix}) is null-controllable if and only if
\begin{equation}\label{Kalman condition with diffusion matrix}
\rank{ \kalm{-\lambda_kJ+A}{D}{n} }=n, \quad \forall k \in \mathbb{N}^*.
\end{equation}
\end{theorem}

When $J$ is the identity matrix condition (\ref{Kalman condition with diffusion matrix}) is indeed equivalent to condition (\ref{distributed Kalman condition}) since $\rank{ \kalm{-\lambda_kI+A}{D}{n} }=\rank{ \kalm{A}{D}{n} }$ for all $k\in \mathbb{N}^*$.

Theorem \ref{Kalman condition with diffusion matrix theorem} can be proved by slightly changing the proof of Theorem 1.1 of \cite{AKBD}. Indeed, with the notations of \cite{AKBD}, changing the definition of $\cal{K}$ one can see that Theorem 1.1 is still a consequence of Theorem 1.3 and Theorem 1.3 is still a consequence of Theorem 1.2 and Theorem 2.1. The proof of Theorem 2.1 remains unchanged and the proof of Theorem 1.2 can be adapted to our case by applying Theorem 3.2 to $\phi=D_i^*\varphi$ instead of $\phi={\left(B^*\varphi \right)}_{i}$.

\item
All the results of section \ref{section n_B} and Theorem \ref{Kalman condition with diffusion matrix theorem} remain true if we replace the operator $-\Delta$ by more general elliptic operators $-R$, of the form
$$\left\{\begin{array}{ll}
& \ds Ry=\sum_{i,j=1}^{N} \partial_i \left( r_{ij}(x) \partial_j y \right),\\
&\ds  r_{ij} \in W^{1,\infty}(\Omega), \quad r_{ij}=r_{ji} \mbox{ in } \Omega, \quad \forall i,j \in \{1,\ldots,N\}, \\
&  \ds \exists \ \underline{r}>0, \quad \sum_{i,j=1}^{N} r_{ij} \xi_i \xi_j \geq \underline{r} {|\xi|}^{2} \mbox{ in } \Omega, \quad \forall \xi \in \mathbb{R}^N.
\end{array}\right.
$$

\item
In this paper we used a particular strategy which does not apply to many other problems. For instance the case of space varying coefficients can not be analyzed the same way. Indeed we did a change of variable so we used the fact that $\Delta$ commutes with $P_1$.

As our result is based on the one of \cite{AKBGBdT}, the $N$-dimensional case with $N>1$ is still open.

\end{enumerate}
\section*{Acknowledgments}

I would like to thank Assia Benabdallah and Franck Boyer for their advices and
remarks.

\bibliographystyle{plain}
\bibliography{null-controllability-parabolic-systems}

\end{document}